\documentclass[11pt]{article}
\usepackage{amsthm, amsmath, amssymb, amsfonts, fancyhdr, verbatim, url}
\usepackage[margin = 1in]{geometry}
\usepackage{graphicx}
\usepackage{setspace}
\usepackage{tikz}
\usepackage{booktabs}

\usepackage{multirow}

\usepackage{hyperref}

\hypersetup{
    pdfstartview={FitH}    
}

\newtheorem*{claim*}{Claim}
\newtheorem*{fact*}{Fact}
\newtheorem{theorem}{Theorem}[section]
\newtheorem*{theorem*}{Theorem}
\newtheorem{proposition}[theorem]{Proposition}
\newtheorem{lemma}[theorem]{Lemma}
\newtheorem*{lemma*}{Lemma}
\newtheorem{corollary}[theorem]{Corollary}

\theoremstyle{definition}
\newtheorem{definition}[theorem]{Definition}

\theoremstyle{remark}

\def\ZZ{\mathbb{Z}}

\newcommand{\abs}[1]{\left\lvert#1\right\rvert}

\newcommand{\floor}[1]{\left\lfloor #1 \right\rfloor}
\newcommand{\ceil}[1]{\left\lceil #1 \right\rceil}

\def\({\left(}
\def\){\right)}

\def\x{\times}

\def\<={\Leftarrow}
\def\=>{\Rightarrow}

\title{Constructing MSTD Sets Using Bidirectional Ballot Sequences}
\author{Yufei Zhao \\[6pt]
\small Department of Mathematics, Massachusetts Institute of Technology\\
\small 77 Massachusetts Ave., Cambridge, MA 02139, USA \\
\small {\tt yufeiz@mit.edu} }

\tikzstyle{P}=[draw,circle, fill=black, minimum size=6pt,inner sep=0pt]

\newcommand{\up}{-- ++(1) node[P]{}}
\newcommand{\down}{-- ++(0) node[P]{}}

\begin{document}

\maketitle

\begin{abstract}
A more sums than differences (MSTD) set is a finite subset $S$ of the integers such that $\abs{S+S} > \abs{S-S}$.
We construct a new dense family of MSTD subsets of $\{0, 1, 2, \dots, n-1\}$. Our construction gives $\Theta(2^n/n)$ MSTD sets, improving the previous best construction with $\Omega(2^n/n^4)$ MSTD sets by Miller, Orosz, and Scheinerman.
\bigskip

\end{abstract}


\section{Introduction}

A more sums than differences (MSTD) set is a finite set $S$ of integers with $\abs{S + S} > \abs{S - S}$, where the sum set $S + S$ and the difference set $S - S$ are defined as
\begin{align*}
	S + S & = \{s_1 + s_2 : s_1, s_2 \in S\} \\
	S - S &= \{s_1 - s_2 : s_1, s_2 \in S\}.
\end{align*}
Since addition is commutative while subtraction is not, two distinct integers $s_1$ and $s_2$ generate one sum but two differences. This suggests that $S+S$ should ``usually'' be smaller than $S - S$. Thus we expect MSTD sets to be rare.

The first example of an MSTD was found by Conway in the 1960's: $\{0,2,3,4,7,11,12,14\}$. The name MSTD was later given by Nathanson \cite{Na}. MSTD sets have recently become a popular research topic \cite{Hegarty, HM, MO, MOS, Na2, Na, Zhao:finite, Zhao:limit}. For older papers see \cite{HRY, Marica, PF, Roesler, Ru1, Ru2, Ru3}. We refer the reader to \cite{Na2, Na} for the history of the problem.

Let $\rho_{n-1} 2^{n}$ be the number of MSTD subsets of $\{0, 1, 2, \dots, n- 1\}$. We refer to $\rho_n$ informally as the \emph{density} of the family of MSTD sets. This quantity was first studied by Martin and O'Bryant \cite{MO}, who showed that $\rho_n \geq 2 \x 10^{-7}$ for $n \geq 14$. However, this bound is far from optimal. Recently, the author \cite{Zhao:limit} showed that $\rho_n$ converges to a limit, and computed a lower bound of $4 \x 10^{-4}$ for this limit. From Monte Carlo experiments, we expect limiting density to be about $4.5 \x 10^{-4}$ \cite{MO}.

The proofs of the lower bounds on $\rho_n$ are non-constructive. On the other hand, infinite families of MSTD sets were constructed by Hegarty \cite{Hegarty}, Nathanson \cite{Na}, and Miller, Orosz, and Scheinerman \cite{MOS}. In particular, Miller et al.~gave the densest construction in terms of the number of subsets of $\{0, 1, \dots, n-1\}$; their construction has density $\Omega(1/n^4)$.

In this paper, we offer a new construction of an infinite family of MSTD sets. Our construction, described in Section \ref{sec:construction}, has density $\Theta(1/n)$, improving the previous result of Miller et al.~\cite{MOS} In Section \ref{sec:bidirectional} we prove that our family of MSTD sets has the claimed size. In the process we introduce a new combinatorial object called bidirectional ballot sequence, whose additional properties are discussed in Section \ref{sec:conclusion}.


\section{Construction of MSTD sets} \label{sec:construction}

We use $[a, b]$ to denote the set $\{a, a+1, \dots, b\}$. In this section we describe our construction of a new family of MSTD subsets of $[0, n-1]$.

The first idea used in our construction is similar to the techniques used in both \cite{MO} and \cite{MOS}; namely we look for sets of the form
\[
	S = L \cup M \cup R,
\]
where
\begin{align*}
L &= S \cap [0, \ell-1], \\
M &= S \cap [\ell, n - r-1], \\
R &= S \cap [n-r, n-1].
\end{align*}
We will fix $L$ and $R$ to be sets with certain desirable properties and let $M$ vary.

\begin{figure}[ht]
	\centering
	$S = $
	\begin{tabular}{|l|l|l|}
	\hline
	 \hspace{.2in}$L$ \hspace{.2in} & \hspace{.4in} $M$ \hspace{.4in} & \hspace{.2in} $R$  \hspace{.2in} \\ 
	\hline
	\end{tabular}
	
	$S + S = $
	\begin{tabular}{|l|l|l|}
	\hline
	 \hspace{.3in}$L+L$ \hspace{.3in} & \hspace{1.3in} ? \hspace{1.3in} & \hspace{.2in} $R+R$  \hspace{.2in} \\ 
	\hline
	\end{tabular}
	
	$S - S = $
	\begin{tabular}{|l|l|l|}
	\hline
	 \hspace{.3in}$L-R$ \hspace{.3in} & \hspace{1.3in} ? \hspace{1.3in} & \hspace{.2in} $R-L$  \hspace{.2in} \\ 
	\hline
	\end{tabular}
	
	\caption{Illustration of the construction of $S$.}
\end{figure}

For instance, adapting the construction from \cite{MO} and taking $\ell = r = 11$ and
\begin{align}
	L &= \{0, 2, 3, 7, 8, 9, 10\} \label{eq:MSTD-L}, \\
	R &= \{n-11, n-10, n-9, n-8, n-6, n-3, n-2, n-1\},  \label{eq:MSTD-R} 
\end{align}
we have
\[
	L + L = [0, 20] \setminus \{1\}, \qquad R + R = [2n - 22, 2n -2].
\]
On the other hand, $S-S$ is missing at least two differences, namely $\pm(n-7)$, so $\abs{S - S} \leq 2n - 3$. If we can get $S+S$ to contain $[21, 2n-23]$ (i.e., all the middle sums not yet covered by $L +L$ or $R + R$), then $S + S$ is only missing the sum $1$, and thus $\abs{S + S} = 2n - 2$, thereby making $S$ an MSTD set.

So our goal is to choose $M$ so that $S + S$ is not missing any sums in the middle segment, i.e., $[21, 2n - 23]$. From the probabilistic argument of \cite{MO}, we know that the set of all $M$'s with this property occupies a positive lower density of all subsets of $[11, n - 12]$. However, that proof is non-constructive.

Note that if $M+M$ is not missing any sums (i.e., $M+M = [2\cdot 11, 2(n - 12)]$), then $S$ has the desired properties. This condition forces $11, n - 12 \in M$, so that $21, 2n-23 \in S + S$ as well. Let us temporarily do some re-indexing so that the problem becomes finding subsets $M$ of $[1, m]$ such that $M + M = [2, 2m]$. Note that the probabilistic argument of \cite{MO} also shows that the set of such $M$'s has at least positive constant density.

The construction of \cite{MOS} is as follows: let $M$ contain all $k$ elements on each of its two ends (i.e., $[1, k] \cup [m-k+1, n] \subset M$), and furthermore let $M$ have the property that it does not have a run of more than $k$ consecutive missing elements. Here $k$ is allowed to vary. This construction gives a density of $\Omega(1/n^4)$.

We use a different approach to construct $M$. The property of $M$ that we seek is the following: for every prefix and suffix of $[1, m]$, more than half of the elements are in $M$. The following lemma proves that this constraint is sufficient for our purposes.

\begin{lemma}
If $M \subset [1, m]$ satisfies 
\[
	\abs{M \cap [1, k]} > \frac{k}{2}, \qquad \text{and} \qquad \abs{M \cap [m-k+1, m]} > \frac k 2
\]
for every $0 < k\leq m$, then $M + M = [2, 2m]$.
\end{lemma}

\begin{proof}
Let $2 \leq x \leq 2m$. If $x \leq m$, then since $M$ contains more than half of the elements in $[1, x-1]$, by the pigeonhole principle, there is some $y$ so that $y, x-y \in M$, so that $x \in M + M$. Similarly, if $x > m$, then since $M$ contains more than half of the elements in $[x-m, m]$, we can find some $x-y, y \in M$ so that $m \in M + M$ as well.
\end{proof}

The construction of this new family of MSTD sets is summarized in the theorem below.

\begin{theorem} \label{thm:MSTD-constr}
Let $n \geq 24$. Moreover, let $M$ be a subset of $[11, n-12]$ with the property that every prefix and every suffix of the interval $[11, n-12]$ has more than half of its elements in $M$. Then $S = L \cup M \cup R$ is an MSTD set, where $L$ and $R$ are given in \eqref{eq:MSTD-L} and \eqref{eq:MSTD-R}. The number of MSTD sets of $\{0, 1, \dots, n-1\}$ in this family is $\Theta(2^n/n)$.
\end{theorem}

To prove the last assertion in the theorem, we need to count the number of sets in our family. This is done in the next section.


\section{Bidirectional ballot sequence} \label{sec:bidirectional}

In order to study the sizes of our new families of MSTD sets, we introduce the following combinatorial construction.

\begin{definition}
A 0-1 sequence of length $n$ is a \emph{bidirectional ballot sequence} if every prefix and suffix contains strictly more $1$'s than $0$'s. The number of bidirectional ballot sequences of length $n$ is denoted $B_n$.
\end{definition}

Recall that a classical \emph{ballot sequence} is a 0-1 sequence where we only require that every prefix has more 1's than 0's. 
A bidirectional ballot sequence is then a ballot sequence whose reverse is also a ballot sequence. This construction appears to be new. Table \ref{tab:B_n} gives some values of $B_n$. At the time of this writing, the sequence $B_n$ was not found on the Sloane On-Line Encyclopedia of Integer Sequences \cite{Sloane}.

\begin{table}[ht!]
\caption{Number of bidirectional ballot sequences of length $n$.\label{tab:B_n}}
\begin{tabular}{ccccccccccccc}
	\toprule
	$n$ & 1 & 2 & 3 & 4 & 5 & 6 & 7 & 8 & 9 & 10 & 11 & 12 \\ 
	
	$B_n$ & 1 & 1 & 1 & 1 & 2 & 3 & 5 & 9 & 15 & 28 & 49 & 91 \\ 
	\midrule
	$n$ & 13 & 14 & 15 & 16 & 17 & 18 & 19 & 20 & 21 & 22 & 23 & 24 \\ 
	
	$B_n$ & 166 & 307 & 574 & 1065 & 2016 & 3769 & 7176 & 13532 & 25842 & 49113 & 93995 & 179775 \\ 
	\bottomrule
\end{tabular}
\end{table}

It is easy to see that the possibilities for the set $M$ in the construction in Theorem \ref{thm:MSTD-constr} correspond bijectively with bidirectional ballot sequences of length $n-22$. Then, the proof of the final assertion in the theorem is equivalent to the following result about the number of bidirectional ballot sequences of a given length.

\begin{proposition} \label{prop:B_n-weak}
The number of bidirectional ballot sequences satisfies $B_n = \Theta\( 2^{n} / n\)$.
\end{proposition}

This rest of this section contains a proof of Proposition \ref{prop:B_n-weak}.

We can interpret 0-1 sequences in terms of lattice walks, where we start at the origin and take steps of the form $(1, 1)$ and $(1, -1)$, corresponding to the terms $1$ and $0$ in the sequence, respectively. Let a \emph{ballot walk} (resp.~\emph{bidirectional ballot walk}) be such a lattice walk corresponding to a ballot sequence (resp.~bidirectional ballot sequence). So, a ballot walk is a lattice walk with the property that the starting point is the unique lowest point, and a bidirectional ballot walk has the additional property that the ending point is the unique highest point. See Figure \ref{fig:halves} for an example.

The key idea in the proof of Proposition \ref{prop:B_n-weak} is to divide a bidirectional ballot walk into two halves, as in Figure \ref{fig:halves}. The second half should be ``reversed,'' i.e., viewed with a $180^\circ$ rotation. For the upper bound, we notice that each half is necessarily a ballot walk. For the lower bound, we need some sufficient condition on the two halves so that neither ``overshoots'' the other when the two halves are glued together.

\begin{figure}
	\centering
	\begin{tikzpicture}[scale = .4]
	\coordinate (1) at (1,1);
	\coordinate (0) at (1,-1);
	\draw[thick] (0,0) node[P] {} 
	 \up\up\down\up\up\down\up\up\down\up\down\down\up\up\up\up\up\down\up\up;
	 \draw (-.5,0) -- ++(21,0);
	 \draw (-.5,8) -- ++(21,0);
	 \draw[red, very thick, dashed] (10, 0) -- ++(0,8);
	\end{tikzpicture}
	\caption{A bidirectional ballot walk corresponding to the sequence $11011011010011111011$.
	         The middle dashed line divides the walk into two halves.\label{fig:halves}}
\end{figure}
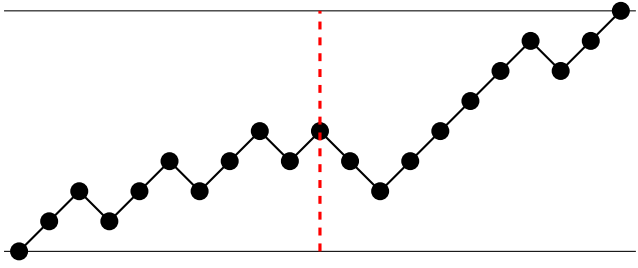

Let us recall the following classic theorem about ballot sequences (e.g., see \cite{ballot}).

\begin{theorem}[Ballot Theorem]
Let $p > q$. The number of ballot sequences with $p$ 1's and $q$ 0's, or equivalently the number of ballot walks with $p$ steps of the form $(1,1)$ and $q$ steps of the form $(1,-1)$, is equal to 
\[ 
	\binom{p+q-1}{p-1} - \binom{p+q-1}{p} = \frac{p-q}{p+q}\binom{p+q}{p}.
\]
\end{theorem}

\begin{corollary} \label{cor:ballot}
Let $0 \leq a < b$ be real numbers. The number of ballot walks with $n$ steps and whose final height is inclusively between $a$ and $b$ is 
\[
	\binom{n-1}{\ceil{\frac12(a + n)} - 1} - \binom{n-1}{\floor{\frac12(b + n)}}. \qedhere
\]
\end{corollary}

\begin{proof}
We use the Ballot Theorem and sum over all $(p, q)$ with $p+q=n$ and $a \leq 2p - n \leq b$ to find that the desired quantity is
\[
	\sum_{a \leq 2p - n \leq b} \( \binom{n-1}{p-1} - \binom{n-1}{p} \) = \binom{n-1}{\ceil{\frac12(a + n)} - 1} - \binom{n-1}{\floor{\frac12(b + n)}}. \qedhere
\]
\end{proof}

We will also use the following well-known fact about the normal approximation of binomial coefficients. It can be proved using either Stirling's formula or the Central Limit Theorem.

\begin{proposition} \label{prop:normal}
For any real number $t$,
\begin{equation} \label{eq:normal}
	\lim_{n \to \infty} \frac{\sqrt n}{2^n}\binom{n}{\frac{1}{2}(n + t \sqrt{n})} = \sqrt{\frac{2}{\pi}} e^{-\frac 1 2 t^2}.
\end{equation}
\end{proposition}

\subsection{Upper Bound}

\begin{lemma} \label{lem:num-ballot-walks}
The number of ballot walks with $n$ steps is  
$\displaystyle \binom{n - 1}{\ceil{n/2} - 1} \sim \frac{2^n}{\sqrt{2\pi n}}$.
\end{lemma}

\begin{proof}
This follows directly from Corollary \ref{cor:ballot} and Proposition \ref{prop:normal}.
\end{proof}

Let $n_0 = \floor{n/2}$ and $n_1 = \ceil{n/2}$. A bidirectional ballot walk is necessarily a ballot walk of length $n_0$ followed by the reverse of a ballot walk of length $n_1$. Therefore, the number of bidirectional ballot walks with $n$ steps is at most
\[
	 O\(\frac{2^{n_0}}{\sqrt {n_0}}\) O\(\frac{2^{n_1}}{\sqrt {n_1}}\) = O\(\frac{2^n}{n}\).
\]
Thus we have proven the following upper bound on $B_n$.

\begin{proposition} \label{prop:lower}
$B_n = O(2^n/n)$.
\end{proposition}

\subsection{Lower Bound}

We know that the first half and the reverse of the second half of a bidirectional ballot walk are both ballot walks, but this alone is not enough to guarantee that the overall walk is a bidirectional ballot walk. So we place additional constraints on each half of the walk.

\begin{definition}
Let $b$ be a positive integer. A $b$-bounded walk is a ballot walk that never goes into the region $y > 2b$ and ends in the region $y > b$.
\end{definition}

\begin{figure}[ht!]
	\centering
	\begin{tikzpicture}[scale=.3]
	\draw (0,0) 
	 -- ++( 45:5) -- ++(-45:2) -- ++( 45:3)  -- ++(-45:4) -- ++( 45:6) -- ++(-45:2) -- ++( 45:4)  -- ++(-45:1) -- ++( 45:1)  
	 -- ++(-45:7) -- ++( 45:4) -- ++(-45:2)  -- ++( 45:3) -- ++(-45:1);
	\draw (0,0) node[label=left:$0$]{} -- ++(31.8,0);
	\draw[dashed] (0,4) node[label=left:$b$]{} -- ++(31.8,0);
	\draw[dashed] (0,8) node[label=left:$2b$]{} -- ++(31.8,0);
	\draw (0,0) -- ++(0,9);
	\end{tikzpicture}
	
	\caption{An example of a $b$-bounded walk.\label{fig:bounded-walk}}
\end{figure}
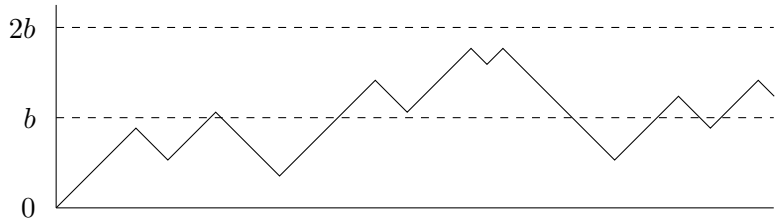

\begin{lemma} \label{lem:catbound}
The concatenation of a $b$-bounded walk followed by the reverse of another $b$-bounded walk is necessarily a bidirectional ballot walk.
\end{lemma}

Figure \ref{fig:proof-catbound} is a ``proof by picture'' of the lemma. The $b$-boundedness ensures that neither half overshoots the other.

\begin{figure}[ht!]
\centering
\includegraphics[scale=.8]{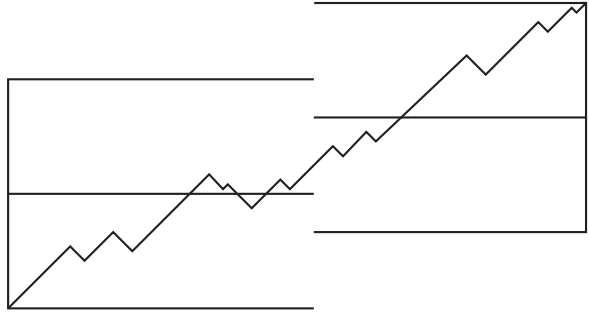}
\caption{``Proof by picture'' of Lemma \ref{lem:catbound}. \label{fig:proof-catbound}}
\end{figure}

\begin{lemma}
The number of $\floor{\sqrt n}$-bounded walks of $n$ steps is $\Omega(2^n/\sqrt n)$.
\end{lemma}

\begin{proof}
We see that $b$-bounded walks of $n$ steps are precisely ballot walks that end in the region $b+1 \leq y \leq  2b$ and never go into the region $y > 2b$. Using Corollary \ref{cor:ballot}, we see that the number of ballot walks with $n$ steps that end in $b+1 \leq y \leq  2b$ is equal to
\[
	\binom{n-1}{ \ceil {\frac12 (n+b-1)}} - \binom{n-1}{ \floor {n/2 } + b}.
\]

Now we need to consider those ballot walks that end in $b < y \leq  2b$ but go into $y > 2b$ at some point in the walk. Let $(t, 2b+1)$ be the last point in walk that is in the region $y > 2b$. We can reflect the portion of the walk after that point to get a ballot walk that ends in $y > 2b + 1$. See Figure \ref{fig:reflect} for an illustration. This map is injective since we can always get back to the original walk, but it is not necessarily onto. Then, we know that the number of ballot walks that end in $b < y \leq  2b$ but go into $y > 2b$ at some point is at most the number of ballot walks that end in $y \geq 2b + 2$. By Corollary \ref{cor:ballot}, the number of ballot walks that end in $y \geq 2b + 2$ is equal to $\binom{n-1}{\ceil{n/2} + b}$.

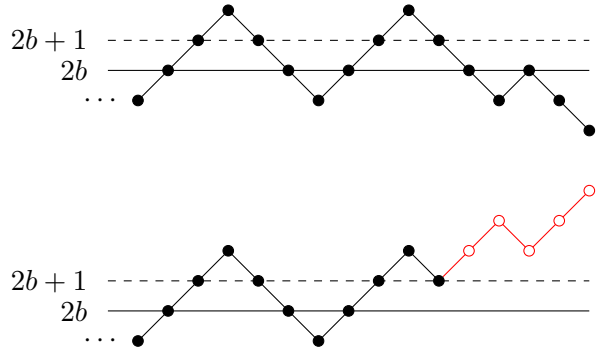
\begin{figure}[ht!] \centering
	\tikzstyle{P}=[draw,circle, fill=black, minimum size=4pt,inner sep=0pt]
	\tikzstyle{Q}=[draw,circle, fill=white, minimum size=4pt,inner sep=0pt]
	\begin{tikzpicture}[scale=.4]
	\coordinate (1) at (1,1);
	\coordinate (0) at (1,-1);
	\begin{scope}	
		\draw (1, -1) node[P, label=left:$\cdots$]{}
			\up\up\up\down\down\down\up\up\up\down\down\down\up\down\down;
		\draw (0,0) node[label=left:$2b$] {} -- (16, 0);
		\draw[dashed] (0,1) node[label=left:$2b+1$] {}  -- ++(16, 0);
	\end{scope}
	
	\begin{scope}[shift={(0,-8)}]
		\coordinate (1) at (1,1);
		\coordinate (0) at (1,-1);
		\draw (1, -1) node[P, label=left:$\cdots$]{}
			\up\up\up\down\down\down\up\up\up
		   -- ++(0) node (x) [P]{};
		\draw[red] (x) -- ++(1) node[Q]{} -- ++(1)node[Q]{} -- ++(0)node[Q]{} -- ++(1)node[Q]{} -- ++(1)node[Q]{};
		\draw (0,0) node[label=left:$2b$] {} -- (16, 0);
		\draw[dashed] (0,1) node[label=left:$2b+1$] {}  -- ++(16, 0);
	\end{scope}
	
	\end{tikzpicture}
	\caption{Reflecting the last segment of a walk. \label{fig:reflect}}
\end{figure}

Therefore, the number of $b$-bounded walks is at least
\[
	\binom{n-1}{ \ceil {\frac12 (n+b-1)}} - \binom{n-1}{ \floor {n/2 } + b} - \binom{n-1}{\ceil{n/2} + b}.
\]
Let $b = \floor{\sqrt n}$. Using Proposition \ref{prop:normal}, we have
\[
	\lim_{n \to \infty}  \frac{\sqrt{n}}{2^n}\( \binom{n-1}{ \ceil {\frac12 (n+b - 1)}} - \binom{n-1}{ \floor {n/2 } + b} - \binom{n-1}{\ceil{n/2} + b}\) = \frac{1}{\sqrt{2\pi}}(e^{-1/2} - 2e^{-2}) > 0.
\]
It follows that the number of $\floor{\sqrt n}$-bounded walks is $\Omega(2^n/\sqrt n)$.
\end{proof}

As before, we can form bidirectional ballot walks by concatenating two $b$-bounded walks, where the second half is reversed. Let $n_0 = \floor{n/2}$ and $n_1 = \ceil{n/2}$. Then, the number of bidirectional ballot walks is at least 
\[
	\Omega\(\frac{2^{n_0}}{\sqrt {n_0}}\)\Omega\(\frac{2^{n_1}}{\sqrt {n_1}}\) = \Omega(2^n/n).
\]
Thus we have proven the following.

\begin{proposition} \label{prop:upper}
$B_n = \Omega(2^n/n)$.
\end{proposition}

Propositions \ref{prop:lower} and \ref{prop:upper} together complete the proof of Proposition \ref{prop:B_n-weak} and hence also Theorem \ref{thm:MSTD-constr}.

\section{Further remarks} \label{sec:conclusion}

We believe that there is more potential to bidirectional ballot sequences than what it presented here. Knowing that $B_n = \Theta\(2^n / n\)$, we can ask whether the ratio $nB_n/2^n$ approaches a limit. Table \ref{tab:ratio} contains some values computed from an exact formula for $B_n$. The data suggest that $nB_n/2^{n-2} \to 1$. This is indeed true. We have a proof of this fact, but our proof is rather long and technical, so we do not present it here. The proof involves first finding an exact formula for $B_n$ using repeated applications of the reflection principle, and then some analysis to estimate the sum. The data in Table \ref{tab:ratio} also suggest the asymptotic expansion
\[
	\frac{B_{n}}{2^n} = \frac{1}{4n} + \frac{1}{6n^2} + O\(\frac{1}{n^3}\),
\]
which we pose as a conjecture.

\begin{table}[ht!] \centering
	\caption{Some values of $nB_n / 2^{n-2}$.\label{tab:ratio}}
	\begin{tabular}{cl}
		\toprule
		$n$ & $n B_n / 2^{n-2}$ \\
		\midrule
		100 & 1.0067268\dots \\
		1000 & 1.00066729\dots \\
		10000 & 1.0000666729\dots \\
		\bottomrule
	\end{tabular}	
\end{table}

Bidirectional ballot sequences look superficially similar to Dyck paths and Catalan numbers. However, the former lack the nice enumerative properties enjoyed by the latter two. There does not seem to be any simple recursive structure in bidirectional ballot sequences, and we were unable to find any useful recurrence relations or generating functions for $B_n$. This is what makes the enumeration of bidirectional ballot sequences particularly difficult.

We can interpret bidirectional ballot sequences in terms of random walks. Suppose we take a random walk of $n$ steps in $\ZZ$ where each step independently moves one unit to the left or the right, each with $1/2$ probability. Let $p_n$ denote the probability that, among all the points visited by the walk, the starting point is minimum and the ending point is maximum. Then $p_n = B_{n+2}/2^n \sim 1/n$ as $n \to \infty$.

Were it the case that $p_n \sim c/n$ for any other constant $c$, then perhaps the result might be much less interesting\footnote{Indeed, if we only require the starting point to be minimum, then it is easy to show that $p_n \sim \sqrt{\frac{2}{\pi n}}$; the constants here are not nearly as nice.}. However, as it stands, we feel that $p_n \sim 1/n$ is not merely a coincidence, and we believe that it deserves a better explanation then the calculation-heavy proof that we have. There should be some natural, combinatorial explanation, perhaps along the lines of grouping all possible walks into orbits of size mostly $n$ under some symmetry, so that almost every orbit contains exactly one walk with the desired property. So far, we do not know of any such explanation.

We are also currently investigating higher dimensional analogues of this type of random walk problems. We have some experimental data that suggest the prevalence of the $1/n$ asymptotics for analogous walks in higher dimensions. We currently have no proof or explanation of this phenomenon.

The asymptotics related to bidirectional ballot sequences are very intriguing, and we hope to generate more interest in these objects.

\section*{Acknowledgments}

This research was carried out at the University of Minnesota Duluth under the supervision of Joseph Gallian with the financial support of the National Science Foundation and the Department of Defense (grant number DMS 0754106), the National Security Agency (grant number H98230-06-1-0013), and the MIT Department of Mathematics. The author would like to thank Joseph Gallian for his encouragement and support. The author would also like to thank Nathan Kaplan and Ricky Liu for reading the paper and making valuable suggestions.



\bibliographystyle{amsplain}
\bibliography{references}

\end{document}